\documentclass[11pt]{amsart}
\usepackage{latexsym, amssymb}
\usepackage{graphicx}
\usepackage{amsthm}
\usepackage{amssymb, amsmath}
\usepackage{amscd}

\def\beq{\begin{equation}}
\def\eeq{\end{equation}}\usepackage{url}
\newtheorem{lem}{Lemma}
\newtheorem{thm}[lem]{Theorem}
\newtheorem{prp}[lem]{Proposition}

\theoremstyle{definition}

\newtheorem{rem}{Remark}

\def\ra{\rightarrow}

\def\beqa{\begin{eqnarray}}
\def\eeqa{\end{eqnarray}}
\def\beqa{\begin{eqnarray}}
\def\eeqa{\end{eqnarray}}

\newtheorem{remark}{Remark}

\DeclareMathOperator{\GL}{GL}
\DeclareMathOperator{\SL}{SL}
\DeclareMathOperator{\SO}{SO}
\def\R{\mathbb{R}}
\def\Z{\mathbb{Z}}

\begin{document}
\title[Flat cycles in the homology of $\Gamma\setminus \SL_m\mathbb R/\SO(m)$]
{Flat cycles in the homology of $\Gamma\setminus \SL_m\mathbb R/\SO(m)$}
\author{Grigori Avramidi, T. T$\hat{\mathrm{a}}$m Nguy$\tilde{\hat{\mathrm{e}}}$n-Phan}
\address{Dept. of Mathematics\\
5734 S. University Avenue\\
Chicago, Illinois 60637}
\email[G.~Avramidi]{gavramid@math.uchicago.edu}
\address{Dept. of Mathematics\\
5734 S. University Avenue\\
Chicago, Illinois 60637}
\email[T. Tam ~Nguyen-Phan]{ttamnp@math.uchicago.edu}
%



\begin{abstract} 
In this paper we show that flat ($m-1$)-dimensional tori give nontrivial rational homology cycles in congruence covers of the locally symmetric space $\SL_m\mathbb Z\setminus \SL_m\mathbb R/\SO(m).$ We also show that the dimension of the 
subspace of $H_{m-1}(\Gamma\setminus \SL_m\mathbb R/\SO(m);\mathbb Q)$ spanned by flat $(m-1)$-tori grows
as one goes up in congruence covers. 

\end{abstract}
\maketitle
\section{Introduction}
Let $M$ be a finite volume, nonpositively curved, locally symmetric manifold.
It is usually difficult to determine the homology of such a manifold. However, 
totally geodesic submanifolds $N$ are natural candidates for non-trivial homology cycles. 
In this paper we study the case where $N$ is a maximal periodic torus of $M$. 
That is $N$ is a compact, totally geodesic, immersed torus 
whose dimension is equal to the geometric (i.e. real) rank of $M$. 
Prasad and Raghunathan have shown \cite{prasadraghunathan} that a locally symmetric space always contains such tori,
while Pettet and Souto have shown that these tori are ``stuck''
in the thick part of the locally symmetric space and cannot be homotoped 
out to the end \cite{pettetsouto}. This leads one to suspect that such tori
might be homologically nontrivial in a strong sense. The main goal of this 
paper is to justify such suspicions 
in the special case when $\Gamma<\SL_m\Z$ is a finite index
torsionfree subgroup and $M=\Gamma\backslash\SL_m\R/\SO(m)$ is the corresponding locally
symmetric space. In this case, maximal periodic tori can be obtained in the following concrete way. 

Let $\tau\in \SL_m\mathbb Q$ be an element whose
characteristic polynomial 
is irreducible and has $m$ distinct real eigenvalues. 
The minset\footnote{The element $\tau$ is a semisimple isometry of the symmetric space $H:=\SL(\mathbb R^m)/\SO(\mathbb R^m)$. The {\it minset} of $\tau$ is the set of points $\{x\in H\mid d_H(x,\tau x)\leq d_H(y,\tau y) \mbox{ for all } y\in H\}$ that are moved the least by $\tau$.} of $\tau$ acting on $H:=\SL_m\mathbb R/\SO(m)$ is a totally geodesic
($m-1$)-dimensional flat $X$ whose image in the quotient space $\SL_m\mathbb Z\setminus \SL_m\mathbb R/\SO(m)$ is 
an isometrically immersed ($m-1$)-dimensional torus. 
We show that such tori yield interesting homology cycles in finite
volume quotients of $H$.  
\begin{thm}
Let $X$ be an $(m-1)$-dimensional flat whose image in $M:=\SL_m\Z\backslash \SL_m\R/\SO(m)$ is compact. 
Then, there is a finite cover $M'$ of $M$ such that the image of $X$ in $M'$ is a non-trivial homology cycle in $H_{m-1}(M';\mathbb Q).$
\end{thm}
The key ideas of the proof of this theorem are the following.
\begin{itemize}
\item[1)] First we find a totally geodesic copy $Y$ of $(\SL_{m-1}\mathbb R/\SO(m-1))\times\R$ in $H$ 
which is defined over $\mathbb Q$ and intersects the flat $X$ transversally (not necessarily orthogonally) in a single point. 
This reduces to showing that the boundaries at infinity of $X$ and $Y$ are linked. 
\item[2)] Then we find a finite index subgroup $\Gamma\leq\SL_m\mathbb Z$ such that the images of $X$ and $Y$ are embedded orientable submanifolds $\overline X$ and $\overline Y$ of $\Gamma\setminus \SL_m\mathbb R/\SO(m)$ intersecting transversally, with all intersection points having the same sign.
\end{itemize}
This means the signed intersection number $\overline X\cap\overline Y$ is non-zero. Since this number does not change when we replace the cycle $\overline X$ by a homologous cycle\footnote{More formally, transverse intersections between compact cycles and closed cycles make sense on the level of homology and give a map $H_{m-1}\times H^{cl}_{m(m-1)/2}\stackrel{\cap}\ra H_0.$}, we conclude that $\overline X$ is a non-trivial homology cycle in $H_{m-1}(\Gamma\setminus \SL_m\mathbb R/\SO(m);\mathbb Q)$.
Similarly, $\overline Y$ is a non-trivial cycle in homology with closed supports $H^{cl}_{m(m-1)/2}(\Gamma\setminus \SL_m\mathbb R/\SO(m);\mathbb Q)$. 

Let $\Gamma$ be a finite index torsion free subgroup of $\SL_m\mathbb Z$ and $\Gamma(p^n):=\Gamma\cap\ker(\SL_m\mathbb Z\ra \SL_m(\mathbb Z/p^n))$ the $p^n$ congruence subgroup. 
The argument sketched above can be generalized to one that uses multiple flats. We prove the following theorem. It shows that 
the subspace of homology generated by flat tori grows as one goes up in congruence covers.
\begin{thm}\label{maintheorem}
Given a prime $p$ and an integer $N$, there is $n_0$ such that for $n\geq n_0$,  the dimension of the subspace of 
$H_{m-1}(\Gamma(p^n)\setminus H;\mathbb Q)$ spanned by flat cycles is $\geq N.$
\end{thm}
\begin{rem}
This also implies nonvanishing for homology in dimensions other than $m-1$. For instance, if $m=3$ then
$\Gamma(p^n)\setminus H$ is homotopy equivalent to a $3$-complex, $b_1(\Gamma(p^n))=0$ 
by the normal subgroup theorem and $\chi(\Gamma(p^n))=0,$ hence $b_3(\Gamma(p^n))=1+b_2(\Gamma(p^n))$ grows
as one goes up in congruence covers.
\end{rem}

\subsection*{Related work}
There is a large and fruitful literature on homology of locally symmetric spaces obtained from totally geodesic submanifolds. Examples are \cite{millson, millsonraghunathan, schwermer, rohlfsschwermer, leeschwermer}. The idea of eliminating unwanted intersections by passing to congruence covers appears in some form in all these works. However, as far as we can tell, the homology studied in those papers comes from cycles which are the fixed point sets of a finite order rational isometry $\sigma$. These cycles are called \emph{special cycles}. One finds another finite order rational isometry $\sigma'$ commuting with $\sigma$ and then intersects the fixed point sets. The resulting fixed point sets intersect orthogonally.

The flat $\mathbb T^{m-1}$-cycles considered in this paper are not of this type. They are not fixed by any finite order isometry. (The flat $X=\mathbb R^{m-1}$ in the universal cover is the fixed set of an abelian group of involutions $(\mathbb Z/2)^m,$ but these involutions are not rational and do not descend to a finite cover.) Further, our complementary subspaces $Y$ do not need to intersect $X$ orthogonally. This gives flexibility in the choice of $Y$ and allows us to find appropriate intersection patterns in the universal cover via a density argument. 

The rationally defined subspaces $Y$ which we intersect with the flats $X$ are more familiar.
For instance, rational copies of $\mathbb H^2\times\mathbb R$ in $\SL_3\mathbb R/\SO(3)$ 
have been studied in \cite{leeschwermer,ash}. 


\subsection*{Acknowledgements}
We would like to thank Juan Souto for asking the first author whether maximal periodic flats in locally symmetric spaces give interesting homology and for suggesting
that intersections might be resolved by passing to finite covers, Wouter van Limbeek for pointing out that
$\GL_2\R$ is not the same thing as $\SL_2\R\times\R$ (making the paper significantly longer), and Cesar Lozano for explaining complex enumerative geometry (which motivated Proposition \ref{projective2}).
\section{The $\SL_3$ case}
In this section we we describe some intersections in the symmetric space $\SL_3\mathbb R/\SO(3)$ in terms of linking on its sphere at infinity and give a ``projective plane'' description of when linking occurs. For the formal argument it is not really necessary to single out the $\SL_3$ case from the general $\SL_m$ case, but in practice the $\SL_3$ case is easier to visualize, so we do it separately and illustrate the argument with some pictures. 
We will describe the sphere at infinity in terms
of flag-eigenvalue pairs in $\mathbb R^3$ (section 2.13.8 in \cite{eberlein}). 
This will allow us to compute intersections of totally geodesic submanifolds $X$ and $Y$ in $\SL_3\mathbb R/\SO(3)$
purely in terms of the sphere at infinity. 

\subsection*{A description of the sphere at infinity $S^4=\partial(\SL_3\mathbb R/\SO(3))$}

The points at infinity correspond to geodesic rays $e^{tZ}$ where $Z$
is a trace zero, symmetric matrix of length one. That is, the eigenvalues $\lambda_1\geq\lambda_2\geq\lambda_3$ arranged
in descending order satisfy
\begin{eqnarray*}
\lambda_1+\lambda_2+\lambda_3&=&0,\\
\lambda_1^2+\lambda_2^2+\lambda_3^2&=&1.
\end{eqnarray*}
Let $E_i=\{v\in\mathbb R^3\mid Zv=\lambda_iv\}$ be the $\lambda_i$-eigenspace of $Z$.
Then, the eigenvalues together with the flag 
$$
0\subset E_1\subset E_1+E_2\subset\mathbb R^3
$$
provide enough information to recover the symmetric matrix $Z.$
This description of the sphere $S^4$ at infinity naturally subdivides it into 
three pieces.
\begin{itemize}
\item A four dimensional piece 
$\lambda_1>\lambda_2>\lambda_3$ and $0\subset E_1\subset E_1+E_2\subset\mathbb R^3$ isomorphic $\mathbb P^3\times(1/\sqrt 6,2/\sqrt 6),$
\item a two dimensional piece $\lambda_1=2/\sqrt6,\lambda_2=\lambda_3=-1/\sqrt 6,$ and $0\subset E_1\subset\mathbb R^3$ isomorphic to 
the real\footnote{All projective spaces in this paper are real} projective plane $\mathbb P_1^2$ of lines through the origin in $\mathbb R^3,$
and
\item a two-dimensional piece $\lambda_1=\lambda_2=1/\sqrt 6,\lambda_3=-2/\sqrt 6,$ and $0\subset E_1+E_2\subset\mathbb R^3$
isomorphic to the projective plane $\mathbb P_2^2$ of planes through the origin in $\mathbb R^3$.
\end{itemize}

\subsection*{The circle $S^1\subset S^4$ corresponding to a flat $\mathbb R^2\subset \SL_3\mathbb R/\SO(3)$}
A two-dimensional flat corresponds to a transverse collection of one-dimensional subspaces $A,B,C\subset\mathbb R^3.$
Running through the possible eigenvalues gives the circle. 

\subsection*{Subgroups $\SL_2\mathbb R\times\mathbb R\subset \SL_3\mathbb R$}
A subgroup $\SL_2\mathbb R\times\mathbb R\subset \SL_3\mathbb R$ corresponds to a pair $(L,P)$
where $L$ is a one-dimensional subspace of $\mathbb R^3$, $P$ is a two-dimensional subspace of $\mathbb R^3,$
and $L$ and $P$ are transverse to each other. 
The {\it flags associated to $(L,P)$} are those that occur on the sphere at infinity of the 
corresponding $\SL_2\mathbb R\times\mathbb R.$ The possible $1$-dimensional subspaces in such
a flag are $L$ or a one-dimensional subspace $L_0\subset P.$ The possible $2$-dimensional subspaces
are $P$ or any two-dimensional subspace $P_0$ containing $L$. The possible flags are nested sequences of these.

Suppose that we have two such subgroups corresponding to pairs $(L,P)$ and $(L',P'),$
and denote the corresponding two-spheres at infinity by $S$ and $S'$ in $S^4$. We describe the points of intersection $S\cap S'.$ 

\subsection*{The intersection $S\cap S'$ in $S^4$}
We say that the pairs $(L,P)$ and $(L',P')$ are in general position if 
$L\not=L',
P\not=P',
L\pitchfork P',
L'\pitchfork P,$ and
$(P\cap P')\pitchfork(L+L').$
\begin{prp}
If the pairs $(L,P)$ and $(L',P')$ are in general position, then the corresponding
spheres $S$ and $S'$ intersect twice. They have no intersections in $\mathbb P^3\times(1/\sqrt 6,2/\sqrt 6)$,
intersect once at 
$P\cap P'\in\mathbb P_1^2$ and once at 
$L+L'\in\mathbb P_2^2.$  
\end{prp}
\begin{proof}
Since $(L,P)$ and $(L',P')$ are in general position, 
the only flags they share are $0\subset P\cap P'\subset\mathbb R^3$ and $0\subset L+L'\subset\mathbb R^3.$
\end{proof}
\subsection*{Transverse intersections}
Let $p$ be a point of intersection of $S$ and $S'$ in $S^4$.
Denote by $L(p,S),L(p,S'),$ and $L(p,S^4)$ the links of $p$ in $S,S'$ and $S^4$, respectively.
If the circles $L(p,S)$ and $L(p,S')$ link in the three-sphere $L(p,S^4),$
then we say that the intersection of $S$ and $S'$ at $p$ is {\it transverse}.

\subsection*{Some real projective geometry}
Think of $A,B,C$ and $L$ as points in $\mathbb P_1^2$ and $P$ a projective line in $\mathbb P_1^2.$
We explain the criterion for determining whether the flat $X$ corresponding to $(A,B,C)$
intersects the copy $Y$ of $\mathbb H^2\times\mathbb R$ corresponding to $(L,P).$
For two distinct points, say $A$ and $B$, we denote by $\overline{AB}$ the projective
line passing through $A$ and $B$. In Figures 1,2, and 3, we draw the projective plane as a disk
with antipodal boundary points identified. 
\begin{prp}
\label{projective}
Suppose that $(A,B,C)$ and $(L,P)$ are in general position.
Then $\mathbb P_1^2\setminus\{\overline{AB},\overline{AC},\overline{BC}\}$
is a union of four open triangles. The spaces $X$ and $Y$ intersect 
if and only if the triangle containing the point $L$ does not meet the line $P$.
The intersection is necessarily transverse. 
\end{prp}
\begin{proof}
\begin{figure}
\centering
\includegraphics[scale=0.65]{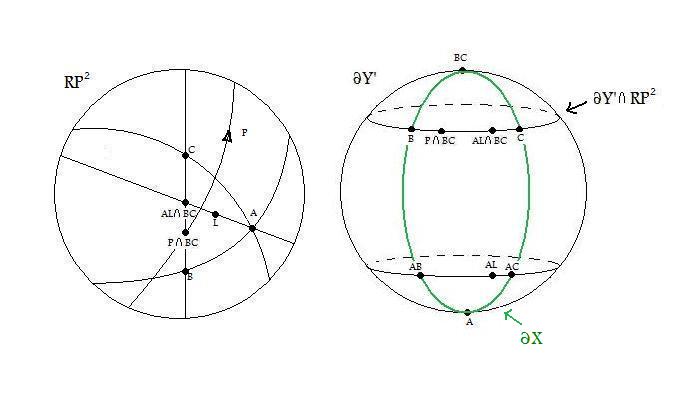}
\caption{}
\end{figure}
\begin{figure}
\centering
\includegraphics*[scale=0.65]{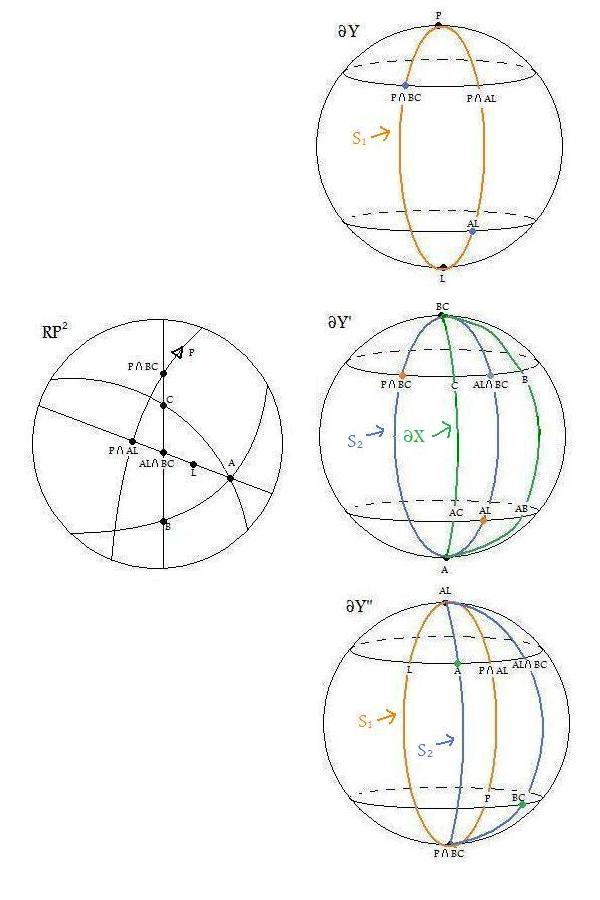}
\caption{}
\end{figure}
Suppose that $P$ passes through the triangle containing $L$. 
For one of the 
vertices of the triangle (without loss of generality, the vertex $A$) the pair of 
points $\{B,C\}$ do not link the pair of points $\{\overline{LA}\cap\overline{BC},P\cap\overline{BC}\}$
on the projective line $\overline{BC}.$ 
We rephrase this in the following way: 
Let $Y'$ be the copy $\mathbb H^2\times\mathbb R$ corresponding to 
$(A,\overline{BC}).$ Then $\partial Y\cap \partial Y'=\{\overline{LA},P\cap\overline{BC}\}$ 
and the non-linking statement becomes the statement that $\partial X$ does not link $\partial Y\cap\partial Y'$ in $\partial Y'.$ 
Thus $\partial X$ does not link $\partial Y$ in the sphere at infinity $S^4$, i.e. $X$ and $Y$ are disjoint.  
This is depicted in Figure 1.

Conversely, suppose that $P$ does not pass through the triangle containing $L$. 
The argument in the previous paragraph shows that $\partial X$ links $\partial Y\cap\partial Y'$ in the sphere $\partial Y'.$ 
To conclude that $\partial X$ links $\partial Y$ in the sphere at infinity $S^4$ we need an additional argument. Namely, 
we need to show that the intersection of the two-spheres $\partial Y$ and $\partial Y'$
at the point $P\cap\overline{BC}$ is transverse.

To show this, we look at the two-sphere at infinity $\partial Y''$ of the symmetric space $Y''$
corresponding to the pair $(P\cap\overline{BC},\overline{LA}).$ A neighborhood
of the point $P\cap\overline{BC}$ in $S^4$ can be identified with a neighborhood of
$(P\cap\overline{BC})\times(P\cap\overline{BC})$  in $\mathbb P^2_1\times\partial Y''.$
Now, $\mathbb P^2_1\cap\partial Y$ is the line $P$ 
together with the point $L$ while $\mathbb P^2_1\cap\partial Y'$ is the line $\overline{BC}$
together with the point $A$.
The lines $P$ and $\overline{BC}$ cross at $P\cap\overline{BC}.$
On the other hand $S_1:=\partial Y''\cap\partial Y$ is 
the boundary of the flat corresponding to $(L,\overline{AL}\cap P,P\cap\overline{BC})$
while $S_2:=\partial Y''\cap\partial Y'$ is the boundary of the flat 
corresponding to $(A,\overline{AL}\cap\overline{BC},P\cap\overline{BC}).$ The circles $S_1$ and $S_2$ in $\partial Y''$ 
cross at $P\cap\overline{BC}$ because the points $\{L,\overline{AL}\cap P\}$ link the points 
$\{A,\overline{AL}\cap\overline{BC}\}$ on the projective line $\overline{AL}.$ 
Since $P$ and $\overline{BC}$ cross in $\mathbb P^2_1$ and $S_1$ and $S_2$ cross in $\partial Y''$
we see that the intersection of $\partial Y$ with $\partial Y'$ at the point $P\cap\overline{BC}$ is transverse.
This second part of the argument is illustrated in Figure 2. 
\end{proof}
\begin{remark}
Corresponding to an open geodesic simplex $\Delta$ in $\mathbb P_1^2$ we can define its dual simplex $\Delta^*$ in the dual projective plane $\mathbb P_2^2$ (with the edges of $\Delta^*$ corresponding to the outside angles of the simplex $\Delta$). Recall that the points $A,B$ and $C$ subdivide $\mathbb P^2_1$ into four simplices $\Delta_1,\Delta_2,\Delta_3$ and $\Delta_4$. 
Then, one can show that Proposition \ref{projective} implies {\it the sphere corresponding to $(L,P)$ links the circle corresponding to $(A,B,C)$ if and only if $(L,P)\in\cup_{i=1}^4\Delta_i\times\Delta^*_i\subset\mathbb P^2_1\times\mathbb P^2_2$.} Since we do not use this description, we will not give a proof of this (but the reader is invited to draw some pictures and convince themselves of it).
\end{remark}
\subsection{An arrangement of intersections}
\begin{figure}
\label{patternpicture}
\centering
\includegraphics[scale=0.6]{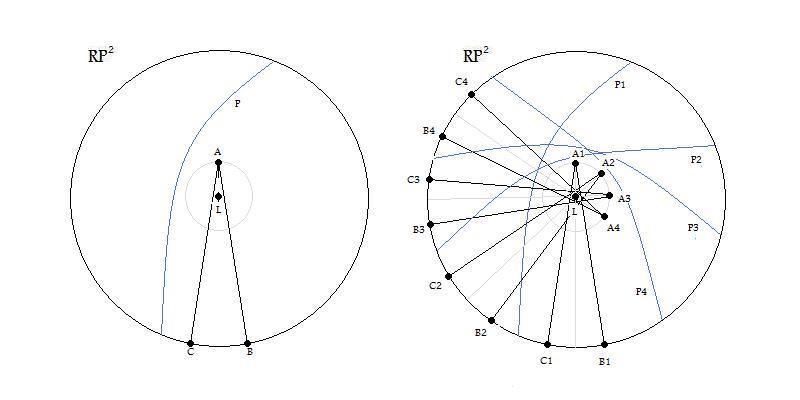}
\caption{}
\end{figure}

\label{intersectionpattern}
Using Proposition \ref{projective} it is easy to construct a pattern $X_1,\dots,X_N,Y_1,\dots,Y_N$
of flats $X_i$ and copies of $\mathbb H^2\times\mathbb R$ denoted $Y_i$, 
for which $X_i$ intersects $Y_j$ if and only if $i\leq j.$ One such pattern (with $N=4$) is indicated in Figure 
3.
The right picture is obtained from the left picture via rotations by a fixed amount. 
To get the required pattern for a general $N$ one starts with a sufficiently thin triangle $(A,B,C)$ and uses a small
enough rotation.

\section{The sphere at infinity of $\SL_m$}
In this section we describe the sphere at infinity in terms of flag-eigenvalue pairs in 
$\mathbb R^m$ (section 2.13.8 in \cite{eberlein}).
Points on the sphere at infinity correspond to geodesic rays $e^{tZ}$ where $Z$ is a trace zero, symmetric matrix
of length one. That is, the eigenvalues
$\lambda_1\geq\dots\geq\lambda_m$ of $Z$ arranged in descending order 
satisfy
\begin{eqnarray*}
\lambda_1+\dots+\lambda_m&=&0,\\
\lambda_1^2+\dots+\lambda_m^2&=&1.
\end{eqnarray*}
Let $E_i:=\{v\in\mathbb R^m\mid Zv=\lambda_iv\}$ be the $\lambda_i$-eigenspace of $Z$.
The eigenvalues together with the flag
$$
0\subset E_1\subset E_1+E_2\subset\dots\subset E_1+\dots+E_{m-1}\subset\mathbb R^m
$$ 
provide enough information to recover the symmetric matrix $Z$. Thus, the points on the sphere
at infinity are parametrized by flag-eigenvalue pairs. The flags
are best thought of as nested arrangements of points, lines, planes etc in real projective space $\mathbb P^{m-1}.$

\subsection*{The sphere of a direct sum decomposition}
Suppose we are given a direct sum decomposition $\mathbb R^m=U_1\oplus\dots\oplus U_r.$
Up to finite index, the subgroup of $\GL(\mathbb R^m)$
preserving this decomposition is $\GL(U_1)\times\cdots\times \GL(U_r)$
and the subgroup of $\SL(\mathbb R^m)$ preserving the decomposition is 
$\mathbb R^{r-1}\times \SL(U_1)\times\cdots\times \SL(U_r).$ 
This group acts on $\SL(\mathbb R^m)/\SO(\mathbb R^m)$ preserving
a totally geodesic symmetric subspace 
\begin{equation}
\label{product}
\mathbb R^{r-1}\times \SL(U_1)/\SO(U_1)\times\cdots\times \SL(U_r)/\SO(U_r).
\end{equation}
We denote the sphere at infinity 
of this symmetric subspace by $S(U_1,\dots,U_r).$ We note that $S(\mathbb R)=\emptyset$ because $\SL(\mathbb R)$
is a point. Generally, if $U$ is an $n$-dimensional vector space then $S(U)=S^{n(n+1)/2-2}.$
The product decomposition (\ref{product}) gives a join decomposition
\begin{equation}
\label{join}
S(U_1,\dots,U_r)=S^{r-2}\star S(U_1)\star\cdots\star S(U_r)
\end{equation} 
for the sphere at infinity. Next, we describe the flags that occur on this sphere.

\subsection*{The sphere $S(U_1,\dots,U_r)$ as the fix set of an element in $\SL(\mathbb R^m)$}
Let $\tau\in\SL(\mathbb R^m)$ be a diagonalizable element whose eigenspace
decomposition is $\mathbb R^m=U_1\oplus\cdots\oplus U_m.$ The minset\footnote{The element $\tau$ is a semisimple isometry of the symmetric space $H:=\SL(\mathbb R^m)/\SO(\mathbb R^m)$. The {\it minset} of $\tau$ is the set of points $\{x\in H\mid d_H(x,\tau x)\leq d_H(y,\tau y) \mbox{ for all } y\in H\}$ that are moved the least by $\tau$.} of $\tau$ is the symmetric
subspace $\mathbb R^{r-1}\times \SL(U_1)/\SO(U_1)\times\cdots\times \SL(U_r)/\SO(U_r)$. 
The action of $\tau$ extends to the sphere at infinity $S(\mathbb R^m)$
and its fixed set is precisely the sphere $S(U_1,\dots,U_r).$ The action of $\tau$ on the sphere
at infinity does not change the eigenvalues and sends a flag $F_1\subset\dots\subset F_k$
to the flag $\tau F_1\subset\dots\subset\tau F_k$ (see 2.13.8 of \cite{eberlein}.)
A standard linear algebra argument shows that the subspaces $V\subset\mathbb R^m$ that are preserved by $\tau$ are precisely those spanned by the eigenvectors of $\tau$. We will call these the subspaces {\it associated} to the eigenspace decomposition $(U_1,\dots,U_r)$. We will say that a flag $F_1\subset\dots\subset F_k$ is {\it associated} to the eigenspace decomposition $(U_1,\dots,U_r)$ if each of the subspaces $F_i$ in the flag is spanned by $\tau$-eigenvectors. These are precisely the flags that are preserved by $\tau$. Thus
\begin{itemize}
\item
The sphere at infinity $S(U_1,\dots,U_r)$ consists of those flag-eigenvalue pairs whose flag is associated to $(U_1,\dots,U_r)$.
\end{itemize}
\subsection*{Projective space description}
Here is a slightly more geometric description.
Let $\mathbb P^{m-1}$ be the projective space of lines through the origin in $\mathbb R^m$. 
The subspaces $U_i$ form a transverse arrangement $(U_1,\dots,U_r)$ of projective subspaces in $\mathbb P^{m-1}.$ 
A $k$-dimensional subspace $V\subset\mathbb R^m$ defines a $(k-1)$-dimensional projective subspace $V\cong\mathbb P^{k-1}\subset\mathbb P^{m-1}$.
It is spanned by the eigenvectors of $\tau$ precisely when there are $k$ points in the union $\cup_{i=1}^r U_i\subset\mathbb P^{m-1}$
that span $V$.  
\subsection*{General position}
Suppose that $L_1,\dots,L_m,L$ are points in $\mathbb P^{m-1}$ and $P$ is an $(m-2)$-dimensional projective subspace $\cong\mathbb P^{m-2}$. 
We say that $(L_1,\dots,L_m)$ and $(L,P)$ are in {\it general position} in $\mathbb P^{m-1}$ if 
\begin{enumerate}
\item
none of the points is contained in $P$,
\item
any $m$ points span an $(m-1)$-simplex in $\mathbb P^{m-1}$, and
\item
$$
\overline{L_iL_j}\cap P\not=\overline{L_iL_j}\cap\overline{LL_1\dots\hat{L_i}\dots\hat{L_j}\dots L_{m}}.
$$
Here, we denote by $\overline{L_iL_j}$ the projective line passing through the 
points $L_i$ and $L_j$ and by $\overline{LL_1\dots\hat{L_i}\dots\hat{L_j}\dots L_{m}}$ the projective hyperplane passing through all the points except for $L_i$ and $L_j$. In words, this third condition says that the intersection of the line passing through the points $L_i$ and $L_j$ with the hyperplane $P$ is not equal to the intersection of that line with the hyperplane passing through the remaining points.
\end{enumerate}
\subsection*{Intersecting spheres at infinity}
Suppose we have two direct sum decompositions 
$$
L_1\oplus\cdots\oplus L_m\cong\mathbb R^m\cong L\oplus P
$$
where $L,L_1,\dots,L_m$ are points and $P$ is a hyperplane in $\mathbb P^{m-1}$.
If $(L_1,\dots,L_m)$ and $(L,P)$ are in general position, then the first two general position conditions imply that
the only subspaces associated to both $(L_1,\dots,L_{m-2},\overline{L_{m-1}L_m})$ and $(L,P)$
are the point $L':=P\cap\overline{L_{m-1}L_m}$ and the hyperplane $Q:=\overline{LL_1\cdots L_{m-2}}.$\footnote{Since $P$ doesn't contain any of the points $L_i$, the only subspace contained in $P$ that is associated to $(L_1,\dots,L_{m-2},\overline{L_{m-1}L_m})$ is $L':=P\cap\overline{L_{m-1}L_m}$. If the subspace is not contained in $P$, then it must  contain $L$, and by general position $Q$ is the smallest dimensional such subspace associated to $(L_1,\dots,L_{m-2},\overline{L_{m-1}L_m})$. Finally, it is easy to see that it is the only such subspace.}
The third general position condition implies that the line $L'$ is not contained in the hyperplane $Q$. Thus, the spheres at infinity
$S(L_1,\dots,L_{m-2},\overline{L_{m-1}L_m})$ and $S(L,P)$ intersect at exactly two points
$L'$ and $Q$ in $S(\mathbb R^m)$.  

\subsection*{A neighborhood of $Q$}
The singleton codimension one flags form a projective space $\mathbb P^{m-1}_{m-2}$
in the sphere at infinity $S(\mathbb R^m).$ This projective space has a regular neighborhood $N(\mathbb P^{m-1}_{m-2})$ 
in $S(\mathbb R^m)$ which is a bundle whose fibre over a point $V\in\mathbb P^{m-1}_{m-2}$ 
can be identified with the cone Cone$(S(V))$ of the sphere at infinity of $V$.  

\section{The linking lemma}
In this section we explain how to compute intersections of totally geodesic submanifolds in the symmetric space $\SL_m\mathbb R/SO(m)$ in terms of linking on the sphere at infinity. We then describe how to determine linking at infinity in terms of the geometry of real projective space. This description turns out to be convenient for constructing and perturbing intersection patterns in the universal cover.

Suppose that $X\cong\mathbb R^{m-1}$ is a flat obtained as the minset of an element $\tau\in\SL_m\mathbb R$
with $m$ distinct real eigenvalues, while $Y$ is a copy of 
$(\SL_{m-1}\mathbb R/\SO(m-1))\times\mathbb R$
which is the minset of an involution $\rho\in\GL_m\mathbb R$ with eigenvalues $(-1,\dots,-1,1).$
Let $(L_1,\dots,L_m)$ be the eigenspaces of $\tau$ and $(L,P)$ be the line-hyperplane pair of eigenspaces of $\rho.$ 
Suppose that $(L_1,\dots,L_m)$ and $(L,P)$ are in general position. Then the spheres at infinity $\partial X=S(L_1,\dots,L_m)$ and $\partial Y=S(L,P)$ are disjoint. Using geodesic projection through a point to the sphere at infinity $S(\mathbb R^m)$ one sees that $X$ and $Y$ intersect if and only if the spheres $\partial X$ and $\partial Y$ link in $S(\mathbb R^m)$. If there is an intersection, then it is necessarily transverse because the spheres $\partial X$ and $\partial Y$ link. 

We give a geometric criterion for determining when the spheres $\partial X$ and $\partial Y$ link. 
\begin{prp}
\label{projective2}
Suppose that $(L_1,\dots,L_m)$ and $(L,P)$ are in general position. Denote by $V_i$ the hyperplane which passes through all the points $L_1,\dots,L_m\in\mathbb P^{m-1}$ except $L_i.$
The hyperplanes $V_1,\dots,V_m$ subdivide $\mathbb P^{m-1}$ into open $(m-1)$-simplices. 
The spheres $S(L_1,\dots,L_m)$ and $S(L,P)$ link if and only if the $(m-1)$-simplex $\sigma$
containing $L$ does not meet $P$.
\end{prp}
\begin{proof}
We do an inductive argument.
The base case $m=2$ is easy. In this case the proposition is simply identifying the projective line
$\mathbb P^1$ with the circle at infinity $\partial\mathbb H^2.$
 
Now, we suppose that the Proposition is known for $m-1$ and prove it for $m$.
Note that either $P$ does not meet the simplex $\sigma$, or it intersects at least one
edge of the simplex. So, without loss of generality we will make the following assumption for the rest of the proof. 
\begin{itemize}
\item
If $P$ meets $\sigma$ then it intersects the edge $E$ of the simplex $\sigma$
lying on the line $\overline{L_{m-1}L_m}$.
\end{itemize}
\begin{remark}
So the projective line $\overline{L_{m-1}L_m}\cong\mathbb P^1$ get broken up into two components: the edge $E$ and its complement.
\end{remark}
If $P$ does not meet $\sigma$ at all, then it intersects the line $\overline{L_{m-1}L_m}$ outside the edge $E$. Let $Q:=\overline{LL_1\cdots L_{m-2}}$
be the hyperplane passing through the points $L,L_1,\dots,L_{m-2}.$ Since $L$ is contained inside the simplex $\sigma$, the hyperplane $Q$ intersects the edge $E$. Thus,
\begin{itemize}
\item
The hyperplane $P$ meets $\sigma$ if and only if on the line $\overline{L_{m-1}L_m}$ the pair of points $\{L_{m-1},L_{m}\}$ link with the points $\{P\cap\overline{L_{m-1}L_m},
Q\cap\overline{L_{m-1}L_m}\}$ 
\end{itemize}

The pair of points $\{L_{m-1},L_m\}=S(L_{m-1},L_m)=S^0$ lies inside the circle $S(\overline{L_{m-1}L_m})=S^1.$
It links the points $\{P\cap\overline{L_{m-1}L_m},Q\cap\overline{L_{m-1}L_m}\}$ in this 
circle if and only if the suspension $S(L_1,\dots,L_m)=S^{m-3}\star S(L_{m-1},L_m)$
links the points $\{P\cap\overline{L_{m-1}L_m},Q\cap\overline{L_{m-1}L_m}\}$ in the suspension 
$S(L_1,\dots,L_{m-2},\overline{L_{m-1}L_m})=S^{m-3}\star S(\overline{L_{m-1}L_m}).$
This is the same as linking the points $\{P\cap\overline{L_{m-1}L_m},Q\}$ because the flags
$Q$ and $Q\cap\overline{L_{m-1}L_m}$ are ``adjacent'': the segment on the sphere at infinity corresponding
to the flag $Q\cap\overline{L_{m-1}L_m}\subset Q$ connects them and does not meet $S(L_1,\dots,L_m)$ so they lie in the same component of 
$S(L_1,\dots,L_{m-2},\overline{L_{m-1}L_m})\setminus S(L_1,\dots,L_m)$.
In summary, 
\begin{itemize}
\item[($\star$)]
\label{linking}
The hyperplane $P$ meets $\sigma$ if and only if 
the sphere $S(L_1,\dots,L_m)$ does not link the pair of points $\{P\cap\overline{L_{m-1}L_m},Q\}$
inside $S(L_1,\dots,L_{m-2},\overline{L_{m-1}L_m})$.
\end{itemize}
To unburden notation slightly, we will from now on denote the sphere $S(L_1,\dots,L_m)$
by the letter $S:=S(L_1,\dots,L_m)$. 
Since $(L_1,\dots,L_m)$ and $(L,P)$ are in general position,
the spheres at infinity $S(L,P)$ and $S(L_1,\dots,L_{m-2},\overline{L_{m-1}L_m})$ 
intersect in precisely the two points $P\cap\overline{L_{m-1}L_m}$ and $Q.$
If $P$ meets the simplex $\sigma$ then ($\star$) shows one of the two connected components of 
$S(L_1,\dots,L_{m-2},\overline{L_{m-1}L_m})\setminus S$ is a ball which is bounded by $S$ and does not intersect $S(L,P).$
This means that the sphere $S$ does not link the sphere $S(L,P)$ which is half of what we needed to show.

Now, suppose that $P$ does not meet the simplex $\sigma.$ It remains to show that in this case
the sphere $S$ links the sphere $S(L,P).$ By ($\star$), in this situation
one of the components of $S(L_1,\dots,L_{m-2},\overline{L_{m-1}L_m})\setminus S$
contains $Q$ but does not contain $P\cap\overline{L_{m-1}L_m}.$ We call this component $D$.
It is a ball with boundary $\partial D=S.$ The sphere $S(L,P)$ does not meet $S$ and intersects
the ball $D$ in a single point $Q$. To show that $S$ and $S(L,P)$ link, it is enough to show that the 
intersection at $Q$ is transverse. Note that
\begin{itemize}
\item
The link of $Q$ in the entire sphere at infinity $S(\mathbb R^m)$ is the join $Lk(Q,\mathbb P^{m-1}_{m-2})\star S(Q)$.
\item 
The link of $Q$ in $D$ is $Lk(Q,\mathbb P^{m-1}_{m-2}\cap D)\star S(L_1,\dots,L_{m-2},Q\cap\overline{L_{m-1}L_m})$.
\item
The link of $Q$ in $S(L,P)$ is $Lk(Q,\mathbb P^{m-1}_{m-2}\cap S(L,P))\star S(L,Q\cap P)$.
\end{itemize}
We will show that
\begin{enumerate}
\item
$Lk(Q,\mathbb P^{m-1}_{m-2}\cap D)$ and $Lk(Q,\mathbb P^{m-1}_{m-2}\cap S(L,P))$ link inside of $Lk(Q,\mathbb P^{m-1}_{m-2})$.
\item
$S(L_1,\dots,L_{m-2},Q\cap\overline{L_{m-1}L_m})$ and $S(L,Q\cap P)$ link inside $S(Q)$. 
\end{enumerate}
which implies that $Lk(Q,D)$ and $Lk(Q,S(L,P))$ link inside $Lk(Q,S(\mathbb R^m))$, and hence the intersection of $D$ and $S(L,P)$ at $Q$ is transverse.
This will complete the proof of the Proposition. 
\begin{enumerate}
\item
Near $Q$, the intersection $\mathbb P^{m-1}_{m-2}\cap D$ is a projective line in $\mathbb P^{m-1}_{m-2}$ (the line of all 
projective hyperplanes in $\mathbb P^{m-1}_1$ passing through the points $L_1,\dots,L_{m-2}$), while $\mathbb P^{m-1}_{m-2}\cap S(L,P)$ 
is a projective hyperplane in $\mathbb P^{m-1}_{m-2}$ (the hyperplane of all hyperplanes passing through $L$ in $\mathbb P^{m-1}_1$).
By general position, the line and hyperplane intersect transversally at $Q$, which shows that $Lk(Q,\mathbb P^{m-1}_{m-2}\cap D)$ and $Lk(Q,\mathbb P^{m-1}_{m-2}\cap S(L,P))$ link inside of $Lk(Q,\mathbb P^{m-1}_{m-2})$.
\item
Second, recall that $L$ is contained in the $(m-1)$-simplex $\sigma,$ the hyperplane $P$ does not meet $\sigma,$
and the hyperplane $Q=\overline{LL_1\cdots L_{m-2}}$ passes through the points $L,L_1,\dots,L_{m-2}.$
Now, we intersect with the hyperplane $Q\cong\mathbb P^{m-2}.$
Notice that $L$ is contained in the $(m-2)$-simplex $Q\cap\sigma$ and the hyperplane $Q\cap P$ in $Q$
does not meet $Q\cap\sigma.$ Further, the simplex $Q\cap\sigma$ has vertices $L_1,\dots,L_{m-2},Q\cap\overline{L_{m-1}L_m}.$
Moreover, $(L_1,\dots, L_{m-2}, Q\cap\overline{L_{m-1}L_m})$ and $(L,Q\cap P)$ are in general position in $Q\cong\mathbb P^{m-2}$.\footnote{The third general position condition is a consequence of the fact that $L$ is contained in the $(m-2)$-simplex $Q\cap\sigma$ while the hyperplane $Q\cap P$ does not meet the simplex $Q\cap\sigma$.} Thus, we can apply the inductive hypothesis to conclude the spheres $S(L_1,\dots,L_{m-2},Q\cap\overline{L_{m-1}L_m})$
and $S(L,Q\cap P)$ link in $S(Q).$
\end{enumerate} 
\end{proof}
\begin{figure}
\label{patternpicture}
\centering
\includegraphics[scale=0.7]{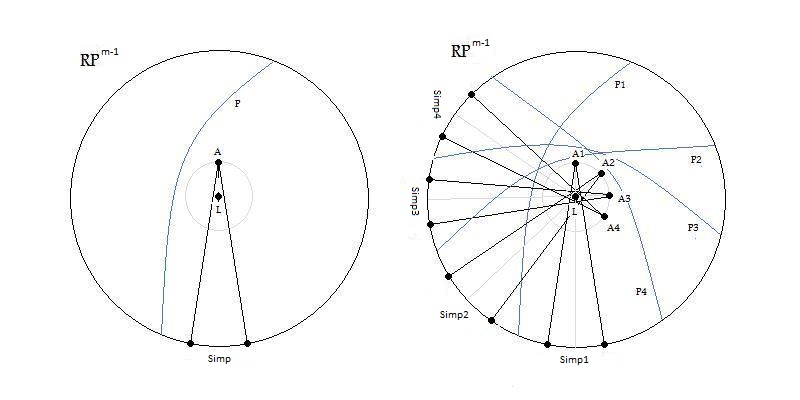}
\caption{}
\end{figure}

\subsection{An arrangement of intersections}
\label{intersectionpattern}
\label{intersectionpattern2}
Using Proposition \ref{projective2} it is easy to construct a pattern $X_1,\dots,X_N,Y_1,\dots,Y_N$
of flats $X_i$ and copies of $(\SL_{m-1}\mathbb R/\SO(m-1))\times\mathbb R$ denoted $Y_i$, 
for which $X_i$ intersects $Y_j$ if and only if $i\leq j.$ One such pattern (with $N=4$) is indicated in Figure 1. 
We draw the ball model of projective space $\mathbb P^{m-1}$, with points on the boundary of the ball identified via the antipodal map.
The right picture is obtained from the left picture via rotations by a fixed amount. 
To get the required pattern for a general $N$ one starts with a sufficiently thin geodesic $(m-1)$-simplex $A\star$Simp and uses a small
enough rotation.

\section{Elements defined over $\mathbb Q$}
Everything we've done so far has been in the symmetric space $\SL_m\mathbb R/\SO(m).$ The
rational structure has not yet entered the picture. It starts to play a role
when one tries to understand how the spaces $X$ and $Y$ project to arithmetic quotients of the
symmetric space. 
\subsection*{Rational flats}
Let $\tau\in \SL_m\mathbb Q$ be an element with $m$ distinct real eigenvalues,
and denote its minset by $X$. It is a totally geodesic submanifold of $H$. 
The centralizer $C_{\tau}(\mathbb R)\cong(\mathbb R^*)^{m-1}$ acts transitively on the minset by orientation preserving isometries. 
The group of all isometries preserving the minset $S_X(\mathbb R):=\{g\in \SL_m\mathbb R\mid gX=X\}$
is the semidirect product $C_{\tau}(\mathbb R)\rtimes S_m$ of the centralizer with the symmetric 
group on $m$ letters (the symmetric group permutes the eigenspaces of $\tau.$) 
Let $\Gamma<\SL_m\mathbb Z$ be a finite index torsionfree subgroup and denote by $\Gamma_X:=\Gamma\cap S_X(\mathbb R)$ 
the group of all isometries in $\Gamma$ preserving the flat $X.$
The following lemma shows that after passing
to a deep enough congruence subgroup we can assume that all isometries of $\Gamma_X$ commute with $\tau.$
\begin{lem}
Fix a prime $p.$ 
Then $\Gamma_X(p^n)\subset C_{\tau}(\mathbb R)$ for sufficiently large $n$. 
\end{lem}
\begin{proof}
Let $\lambda_1,\dots,\lambda_m$ be the eigenvalues of $\tau$
and $K:=\mathbb Q(\lambda_1,\dots,\lambda_m)$ the field obtained by adjoining those eigenvalues. 
The group $S_X(K)$
is $\SL_m K$-conjugate to $(K^*)^{m-1}\rtimes S_m,$
where the $(K^*)^{m-1}$ consists of determinant one diagonal matrices and the symmetric group $S_m$ is represented by permutation matrices. 
From this it follows that
$S_X(K)$ decomposes into $C_{\tau}(K)$-cosets 
$$
S_X(K)=C_{\tau}(K)\cup C_{\tau}(K)\gamma_1\cup\dots\cup C_{\tau}(K)\gamma_r,
$$ 
with the matrices $\gamma_i$ being conjugates of the permutation matrices, i.e. lying in 
$\SL_mK.$ Let $K_p:=\mathbb Q_p(\lambda_1,\dots,\lambda_m)$ be the $p$-adic completion. 
Note that the non-identity cosets lie in the closed subset
$$
\bigcup_{i=1}^r\{x\in \SL_m K_p\mid[x\gamma_i^{-1},\tau]=1\},
$$
of $\SL_m K_p.$
This subset does not contain $1$ since $[\gamma^{-1}_i,\tau]\not=1,$ so there
is a small $p$-adic neighborhood of the identity $U(p^n)$ where all elements from $S_X(K)$ 
commute with $\tau$ i.e.
$$
\Gamma_X(p^n)=\Gamma_X\cap U(p^n)\subset S_X(K)\cap U(p^n)\subset C_{\tau}(\mathbb R). 
$$
\end{proof}
Since $\tau$ is a matrix with entries in $\mathbb Q$, the image of $X$ in $H/\Gamma$ is an isometrically immersed 
($m-1$)-dimensional flat. (See Theorem D in \cite{schwermer} for a proof of this.)
Let $p_{\tau}(t)=\det(t-\tau)$ be the characteristic polynomial.
The following is a special case of a theorem of Prasad and Raghunathan in \cite{prasadraghunathan}:
\begin{prp}
\label{compact}
Suppose $\tau\in \SL_m\mathbb Q$ has $m$ distinct real eigenvalues and irreducible characteristic polynomial.
Then $X/\Gamma_X$ is compact and finitely covered by a ($m-1$)-dimensional torus $\mathbb T^{m-1}.$
\end{prp}

\subsection*{Rational copies of $(\SL_{m-1}\mathbb R/\SO(m-1))\times\mathbb R$}
Now, let $\rho\in \GL_m\mathbb Q$ be a diagonalizeable matrix with eigenvalues $(-1,\dots,-1,1).$ 
Then the centralizer $C_\rho(\mathbb R)\cong \GL_{m-1}\mathbb R$
acts on the minset $Y\cong\SL_{m-1}\mathbb R/\SO(m-1))\times\mathbb R$ of 
$\rho$, and the quotient $Y/\Gamma_Y$
is a properly immersed submanifold of $H/\Gamma.$ (Theorem D in \cite{schwermer}.) In this case the group $S_Y(\mathbb R)$ of 
isometries preserving $Y$ is just the centralizer $C_{\rho}(\mathbb R).$
Note that $C_{\rho}(\mathbb R)\cong \GL_{m-1}\mathbb R$ has two components.
If $m$ is even, then the entire centralizer preserves the orientation of $Y$, but if $m$ 
is odd, then the elements not in the identity component do not preserve the orientation of $Y$.
The eigenspaces of $\rho$ are a (rational!) line $L$ and a hyperplane $P$. 
Whether or not an element preserves orientation can be determined by its action on the line $L$.
(This is noted in the discussion after Corollary 2.4 of \cite{leeschwermer}.) 
Let $\gamma\in C_{\rho}(\mathbb Z)$ be an element in the centralizer with integer entries. Then $\gamma$ preserves the rational line $L$, so $\gamma\mid_L$ is multiplication by some $\lambda\in\mathbb Q$. In fact, $\lambda=\pm 1$:
\begin{proof}
Write the characteristic polynomial of $\gamma$ as
\begin{equation}
\label{characteristic}
\det(t-\gamma)=(t-\lambda_1)\cdots(t-\lambda_m)=t\cdot p(t)\pm 1
\end{equation}
where $p(t)$ is a monic polynomial with integer coefficients.  
Since $\lambda$ is a root of this monic polynomial with integer coefficients, it is an algebraic integer. Moreover, the right hand side of (\ref{characteristic}) shows that the inverse $\lambda^{-1}=\pm p(\lambda)$ is also an algebraic integer. Since $\lambda\in\mathbb Q$ is a rational number, this shows that both $\lambda$ and $\lambda^{-1}$ are integers (a rational algebraic integer is just an integer), which can only happen if $\lambda=\pm 1$.
\end{proof}
Further, the element $\gamma$ preserves an orientation on $Y$ precisely when we have $+$ signs, i.e. $\det\gamma\mid_P=\det\gamma\mid_L=1.$
Let $0\not=v\in L_{\mathbb Z}\subset\mathbb Z^m\subset\mathbb Q^m$ be a non-zero vector with 
integer entries. Then, $\gamma v=\pm v.$ 
For sufficiently large $n$ we have $v\not=-v$ in the quotient $L_{\mathbb Z/p^n}\subset(\mathbb Z/p^n)^m$ so that the 
subgroup $\Gamma_Y(p^n)$ which acts trivially on $L_{\mathbb Z/p^n}$ must preserve the orientation of $Y$.

\subsection*{Embeddings vs immersions} 
A result of Raghunathan (Theorem E in \cite{schwermer}) shows that we can always 
find a positive integer $K_0$ such that for $K\geq K_0$
the maps $X/\Gamma_X(K)\ra H/\Gamma(K)$ and $Y/\Gamma_Y(K)\ra H/\Gamma(K)$ are embeddings. 
Thus, we can replace the group $\Gamma$ by the congruence subgroup $\Gamma(p^n),p^n\geq K_0$ to make 
sure that the quotients $X/\Gamma_X(p^n)$ and $Y/\Gamma_Y(p^n)$ are embedded in the quotient $H/\Gamma(p^n)$.
We remark that this is the same as saying that $\Gamma(p^n)X$ is a disjoint union of copies of $X$ in $H$, and
$\Gamma(p^n)Y$ is a disjoint union of copies of $Y$ in $H$.

We summarize the conclusions of this section in the following proposition, and add an extra bullet. 
The notations are the same as the ones used throughout the section.
\begin{prp}
\label{reduction}
Let $p$ be a prime. Then for sufficiently large $n$, 
\begin{itemize}
\item
$\Gamma(p^n)X$ is a disjoint union of copies
of $X$ and $\Gamma(p^n)Y$ is a disjoint union of copies of $Y$. 
\item
The subgroup $\Gamma_X(p^n)$ centralizes $\tau$ and the subgroup $\Gamma_Y(p^n)$ centralizes $\rho.$
\item
There are $\Gamma(p^n)$-invariant orientations on $\Gamma(p^n)X$ and $\Gamma(p^n)Y$.
\end{itemize}
\end{prp}
\begin{proof}
Everything except for the third point has already been proved in the section above.
Further, we've shown that $\Gamma_Y(p^n)$ preserves an orientation on $Y$ for large enough $n$.
For such $n$ we can---starting with an orientation $Y^+$ of $Y$---define a $\Gamma(p^n)$-invariant orientation
on $\Gamma(p^n)Y$ by $(\gamma Y)^+=\gamma Y^+$ for $\gamma\in\Gamma_Y(p^n)$. 

For large enough $n$, the group $\Gamma_X(p^n)$ is contained in the centralizer $C_{\tau}(\mathbb R),$
and this centralizer preserves an orientation on $X$. Hence, we can define an invariant orientation on $\Gamma(p^n)X$ in
the same way as for $Y.$
\end{proof}

\section{Rational intersection patterns}
In subsection \ref{intersectionpattern2} we described a certain intersection pattern $\{X_i,Y_i\}_{i=1}^N$
involving finitely many flats and copies $(\SL_{m-1}\mathbb R/\SO(m-1))\times\mathbb R$. 
In this section we explain how to get the same intersection pattern using using rational flats that are compact in the quotient,
and copies of $(\SL_{m-1}\mathbb R/\SO(m-1))\times\mathbb R$ that are defined over $\mathbb Q.$ 

\subsection*{Intersection patterns are preserved by small perturbations}
From the description of intersections between $X_i$ and $Y_i$ given Proposition \ref{projective2},
it is clear that an intersection pattern does not change when $X_i$ is perturbed to a nearby flat $X'_i,$
(so the $m$-tuple of eigenspaces is perturbed slightly). Similarly, the intersection pattern does not
change when $Y_i$ is perturbed to a nearby $Y'_i$ (so the line-hyperplane pair $(L,P)$ is perturbed slightly.) 

\subsection*{Small rational perturbations of $X$ and $Y$ exist}
Note that the $\SL_m\mathbb Q$-orbit of transverse $m$-tuple $(L_1,\dots,L_m)$ is dense
in the space of all such triples. In particular let $\tau$ be an element with $m$
distinct real eigenvalues and irreducible characteristic polynomial. The orbit of the $m$-tuple of eigenspaces of $\tau$
is dense. Thus, any flat $X$ has arbitrarily small perturbations whose quotients are compact $(m-1)$-dimensional tori.
Similarly, the rational line-hyperplane pairs are dense in the space of all such pairs, so $Y$ has 
arbitrarily small perturbations whose quotients are properly immersed.

Putting these two observations together with the intersection pattern described in subsection \ref{intersectionpattern2},
we get the following
\begin{prp}
\label{rationalintersectionpattern}
There are rational flats $\{X_i\}_{i=1}^N$ whose quotients $X_i/\Gamma_{X_i}$ are compact 
and rational copies of $(\SL_{m-1}\mathbb R/\SO(m-1))\times\mathbb R$ denoted $\{Y_i\}_{i=1}^N$ whose quotients $Y_i/\Gamma_{Y_i}$
are properly immersed in $H/\Gamma$ such that $X_i$ intersects $Y_j$ if and only if $i\leq j$. Further, all the intersections of the
$X_i$ and the $Y_j$ are transverse.  
\end{prp}

\section{Pushing intersection patterns down to a congruence cover}
The goal of this section is to show that the intersection pattern described in
Proposition \ref{rationalintersectionpattern} can be pushed down to appropriate
congruence covers. Once this is done, we will be able to conclude that the submanifolds involved in the intersection pattern give nontrivial
homology cycles in those congruence covers. Our method is to ``make intersections more similar to each other'' by passing to congruence covers.  
\begin{thm}
\label{covers}
Let $p$ be a prime. Suppose that $\tau\in \SL_m\mathbb Q$ is a matrix with
$m$ distinct real eigenvalues and irreducible characteristic polynomial,
while $\rho\in \GL_m\mathbb Q$ is a diagonalizable matrix with eigenvalues $(-1,\dots,-1,1).$
Let $X$ and $Y$ be the minsets of $\tau$ and $\rho,$ respectively.
Suppose further that we have the following ``genericity'' condition
\begin{itemize}
\item[($t$)]
\label{generic}
The only solutions to the system of equations $\rho x=x\rho,\tau x=x\tau$ are the scalar matrices\footnote{Note that this is a linear system of equations with rational coefficients, so it is equivalent over $\mathbb R,\mathbb Q$ or $\mathbb Q_p$.}. 
\end{itemize}
Then
\begin{enumerate}
\item 
If $X$ and $Y$ intersect transversely in a single point $z\in H,$
then for sufficiently large $n$ the quotients $X/\Gamma_X(p^n)$ and $Y/\Gamma_Y(p^n)$ are embedded, orientable submanifolds,
all their intersections in $H/\Gamma(p^n)$ are transverse and have the same sign. 
\item
If $X$ and $Y$ are disjoint then for sufficiently large $n$ the quotients $X/\Gamma_X(p^n)$ and $Y/\Gamma_Y(p^n)$
are disjoint in $H/\Gamma(p^n).$ 
\end{enumerate}
\end{thm}

\section{Proof of theorem \ref{covers}}
Proposition \ref{reduction} shows that, replacing $\Gamma$ by a sufficiently deep congruence subgroup $\Gamma(p^n)$ if necessary,
we may assume $\Gamma X$ is a disjoint union of copies $X$, equipped with a $\Gamma$-invariant
orientation (so that if $X^+$ is the oriented flat then $(\gamma X)^+=\gamma X^+$
for all $\gamma\in\Gamma$) and $\Gamma_X\subset C_{\tau}(\mathbb R),$ i.e. all elements
of $\Gamma$ preserving the flat $X$ commute with the element $\tau.$ Similarly for the subspace $Y$.
Then, the quotient $X/\Gamma_X$ is a compact oriented submanifold of $H/\Gamma$ and $Y/\Gamma_Y$ is a closed oriented submanifold of $H/\Gamma$. The intersection of these two submanifolds has finitely many components, and each of these components corresponds to a $\Gamma_Y$-orbit of components in $\Gamma X\cap Y$. Our goal in this proof will be to show that if we replace $\Gamma$ by a sufficiently deep congruence subgroup $\Gamma(p^n)$, then all the intersections $\Gamma(p^n)X\cap Y$ have a very special type. 
\subsection*{Claim}
There is $n$ (depending only on the pair of elements $\tau,\rho\in\SL_m\mathbb Q$ and the prime $p$) such that for any $\gamma\in\Gamma(p^n),$ either 
\begin{itemize}
\item
$\gamma X$ and $Y$ don't intersect, or 
\item
$\gamma$ factors as a product $\gamma=ab$ where $a\in C_{\rho}(\mathbb Q)$ preserves $Y$ and the orientation on $Y$ and $b\in C_{\tau}(\mathbb Q)$ preserves $X$ and the orientation on $X$. 
\end{itemize}
Given this claim, we get
$$
\gamma X^+\cap Y^+=abX^+\cap Y^+=aX^+\cap Y^+=a(X^+\cap a^{-1}Y^+)=a(X^+\cap Y^+).
$$
This means all the intersections are transverse and have the same sign as the intersection
$X^+\cap Y^+$ (since $a(X^+\cap Y^+)$ is a shifted version of the transverse intersection $X^+\cap Y^+$ and the shift $a\in \SL_m\mathbb Q$ preserves the orientation of the entire symmetric space $H$.) This implies the first part of the theorem.
Further, it means that if $X$ and $Y$ do not intersect, then $\gamma X$ and $Y$ do not intersect.
This gives the second part of the theorem.

We prove the claim in four steps.
\begin{enumerate}
\item(Reduction to $p$-adic centralizers)
There is an $n$ (depending on the pair $\tau,\rho$ and the prime $p$) such that for any $\gamma\in\Gamma(p^n)$ if $\gamma X$ and $Y$ intersect then $\gamma$ can be expressed as a product $\gamma=a'b'$ where $a'\in C_{\rho}(\mathbb Z_p)$ and $b'\in C_{\tau}(\mathbb Z_p)$.
\item(A linear algebraic lemma)
There are rational matrices $a,b\in\GL_m\mathbb Q$ and a scalar $c\in\mathbb Q_p$ such that $a=ca'$ and $b=c^{-1}b'$.
\item(Product lemma)
The product of centralizers $C_{\rho}(\mathbb Z_p)C_{\tau}(\mathbb Z_p)$ is finitely covered by the Cartesian product $C_{\rho}(\mathbb Z_p)\times C_{\tau}(\mathbb Z_p)$.
\item
If $n$ was taken to be sufficiently large (again, depending only on $\tau,\rho$ and $p$) then for any $\gamma\in\Gamma(p^n)$ there are rational matrices of determinant one $a\in C_{\rho}(\mathbb Q)$ and $b\in C_{\tau}(\mathbb Q)$, such that $\gamma=ab$, $a$ preserves the orientation of $Y$ and $b$ preserves the orientation of $X$.
\end{enumerate}
Recall that the quotient $X/\Gamma_X$ is compact (by Proposition \ref{compact}), so the intersection of $X/\Gamma_X$ and $Y/\Gamma_Y$ in $H/\Gamma$ has 
finitely many components. Consequently, the flats in $\Gamma X$ that intersect $Y$ break up into finitely many $\Gamma_Y$-orbits which we will denote by
$\{\Gamma_Y\gamma_iX\}_{i=0}^k, \gamma_i\in\SL_m\mathbb Z.$ 

\subsection*{Step 1}
We begin with several remarks about the $p$-adic centralizers $C_{\rho}(\mathbb Z_p)$ and $C_{\tau}(\mathbb Z_p)$. Roughly, one can think of these centralizers as the ``$p$-adically closed versions of $\Gamma_Y$ and $\Gamma_X$'', respectively. 
\begin{itemize}
\item
The centralizers $C_{\tau}(\mathbb Z_p)$ and $C_{\tau}(\mathbb Z_p)$ are compact in the $p$-adic topology, so their product $C_{\rho}(\mathbb Z_p)C_{\tau}(\mathbb Z_p)$ is also compact. Moreover, the groups $\Gamma_X$ and $\Gamma_Y$ consist of integer matrices that commute with $\tau$ and $\rho$, respectively, so $\Gamma_Y\Gamma_X$ is contained in $C_{\rho}(\mathbb Z_p)C_{\tau}(\mathbb Z_p)$. 
\item
If an element $\gamma\in\SL_m\mathbb Z$ is not contained in the closed set $C_{\rho}(\mathbb Z_p)C_{\tau}(\mathbb Z_p)$ then a small enough $p$-adic neighborhood of $\gamma$ misses this closed set, and hence also misses the product $\Gamma_Y\Gamma_X$. In other words, $\Gamma(p^n)\gamma\cap\Gamma_Y\Gamma_X=\emptyset$ for large enough $n$.
\end{itemize}
\begin{lem}
If $\Gamma(p^n)\gamma\cap\Gamma_Y\Gamma_X=\emptyset$ then none of the flats $\Gamma_Y\gamma X$ occur in $\Gamma(p^n)X$.
\end{lem}
\begin{proof}
Suppose there are elements $a\in\Gamma_Y$ and $c\in\Gamma(p^n)$ such that $a\gamma X=cX$. Then $a(a^{-1}c^{-1}a)\gamma X=X$ which means $a(a^{-1}c^{-1}a)\gamma=:b$ is an element of $\Gamma_X$. The equation $(a^{-1}c^{-1}a)\gamma=a^{-1}b$ shows that $\Gamma(p^n)\gamma\cap\Gamma_Y\Gamma_X\not=\emptyset$, which is a contradiction. 
\end{proof}
Now, pick an $n$ so that for every $\gamma_i\notin C_{\rho}(\mathbb Z_p)C_{\tau}(\mathbb Z_p)$ we have 
$\Gamma(p^n)\gamma_i\cap\Gamma_Y\Gamma_X=\emptyset$. If $\gamma\in\Gamma(p^n)$ and the flat $\gamma X$ intersects $Y$, then it must have the form $a\gamma_iX$ with $a\in\Gamma_Y$ and $\gamma_i\in C_{\rho}(\mathbb Z_p)C_{\tau}(\mathbb Z_p)$. So $\gamma=a\gamma_ib$ for some $b\in\Gamma_X$, which implies that $\gamma$ is in the product of centralizers $C_{\rho}(\mathbb Z_p)C_{\tau}(\mathbb Z_p)$. This completes the proof of Step 1. 

\begin{remark}
It is easy to check that $\Gamma(p^n)\gamma\cap\Gamma_Y\Gamma_X=\emptyset \iff \Gamma(p^n)\cap\Gamma_Y\gamma\Gamma_X=\emptyset$.
So we get the following interpretation: If the double coset $\Gamma_Y\gamma\Gamma_X$ misses a $p$-adic neighborhood of the identity $\Gamma(p^n)$, then 
the flats $\Gamma_Y\gamma X$ do not appear in $\Gamma(p^n)X$. We will use this later in Step 4. 
\end{remark}

\begin{figure}
\label{padic}
\centering
\includegraphics[scale=0.75]{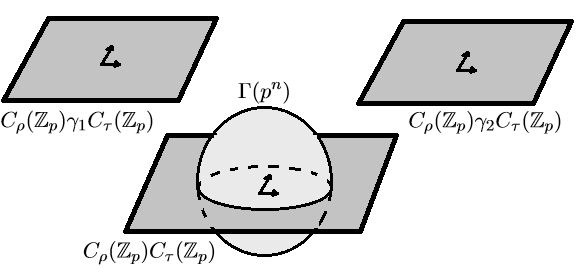}
\caption{}
\end{figure}

Above in Figure 5 is a (very) heuristic $p$-adic picture illustrating Step 1. The double cosets $\{C_{\rho}(\mathbb Z_p)\gamma_iC_{\tau}(\mathbb Z_p)\}_{i=0}^k$ are all compact and locally look like products\footnote{We will prove this local product structure in Step 3 and use it in Step 4.}. A small enough $p$-adic neighborhood of the identity $\Gamma(p^n)$ intersects $C_{\rho}(\mathbb Z_p)C_{\tau}(\mathbb Z_p)$ and does not meet any of the other (finitely many) non-identity double cosets.

\subsection*{Step 2} The following linear algebraic lemma lets us go from matrices over $\mathbb Z_p$ back to ordinary rational matrices (over $\mathbb Q$), at the expense of allowing the determinants to vary.
\begin{lem}
\label{ptoq}
If $\gamma=a'b'$ where $a'\in C_{\rho}(\mathbb Z_p)$ and $b'\in C_{\tau}(\mathbb Z_p)$ then there are rational matrices $a,b\in \GL_m\mathbb Q$ and a scalar $c\in\mathbb Q_p$ such that $a=ca',b=c^{-1}b'$.  
\end{lem}
\begin{proof}
The element $a'$ satisfies the equations $[a',\tau]=[\gamma,\tau]$ and $[a',\rho]=1.$
These equations can be rewritten as
\begin{eqnarray}
\label{linear1}
a'\tau&=&[\gamma,\tau]\tau a',\\
\label{linear2}
a'\rho&=&\rho a'.
\end{eqnarray}
These equations are linear homogeneous in the matrix entries of $a'$ and are defined over $\mathbb Q$ (since 
$\gamma,\tau,$ and $\rho$ are defined over $\mathbb Q$). If they have a $p$-adic solution $a'$
of determinant one, then they also have a rational solution $a$ with non-zero determinant. Moreover, it is easy to check that $a^{-1}a'$ commutes with both $\rho$ and $\tau$, so by the genericity condition ($t$) 
it is a scalar matrix\footnote{In other words, the genericity condition implies the vector space of solutions is one-dimensional.}, which means $a=ca'$ for some constant $c\in\mathbb Q_p$. 
Take $b=a^{-1}\gamma.$ 
\end{proof}

\subsection*{Step 3}
In this step we relate the product of centralizers to their Cartesian product. This will make it a bit easier to separate a double coset $\Gamma_Y\gamma_i\Gamma_X$ from a small $p$-adic neighborhood of the identity in the next (last) step.

Let $\mu:=\{\nu\in\mathbb Q_p\mid\nu^m=1\}$ be the group of $m$-th roots of unity in $\mathbb Q_p$. Since these roots of unity have finite order, they all actually lie in $\mathbb Z_p$. The genericity condition ($t$) 
says the only matrices that commute with both $\rho$ and $\tau$ are the scalar matrices. These matrices have determinant one precisely if the scalar is an $m$-th root of unity. In other words, the intersection of centralizers $C_{\rho}(\mathbb Z_p)\cap C_{\tau}(\mathbb Z_p)$ is the group $\mu$.
\begin{lem}
The map 
\begin{eqnarray*}
C_{\rho}(\mathbb Z_p)\times C_{\tau}(\mathbb Z_p)&\ra&C_{\rho}(\mathbb Z_p)C_{\tau}(\mathbb Z_p),\\
(x,y)&\mapsto& xy,
\end{eqnarray*}
is a finite sheeted regular cover with covering group $\mu$ acting via $\nu\cdot(x,y)=(\nu x,\nu^{-1}y)$.
\end{lem}
\begin{proof}
Condition ($t$) implies that the map in the statement of the lemma is finite-to-one and its fibers are the $\mu$-orbits. Quotienting out by $\mu$, we get a map 
$$
(C_{\rho}(\mathbb Z_p)\times C_{\tau}(\mathbb Z_p))/\mu\ra C_{\rho}(\mathbb Z_p)C_{\tau}(\mathbb Z_p).
$$ 
This map is a continuous bijection of compact Hausdorff spaces, so it is a homeomorphism\footnote{But not a group homomorphism. The right hand side isn't even a group.}.
On the other hand, the group $\mu$ acts on the Cartesian product $C_{\rho}(\mathbb Z_p)\times C_{\tau}(\mathbb Z_p)$ by covering translations\footnote{Officially, it acts {\it freely, properly discontinuously}.} so $C_{\rho}(\mathbb Z_p)\times C_{\tau}(\mathbb Z_p)\ra (C_{\rho}(\mathbb Z_p)\times C_{\tau}(\mathbb Z_p))/\mu$ is a finite sheeted regular cover.
\end{proof} 

\subsection*{Step 4}The following lemma completes the proof of the main claim. The main point of the lemma is to look at how matrices commuting with $\rho$ act on the $1$-eigenspace $L$ of $\rho$ and throw away the intersection by passing to a finite cover if $L$ gets rescaled in a non-trivial way. 
\begin{lem}
\label{orient}
For sufficiently large $n$, all the flats in $\Gamma(p^n)X$ which intersect $Y$ 
are of the form  
$abX$ where $a\in C_{\rho}(\mathbb Q)$, $b\in C_{\tau}(\mathbb Q),$
$a$ preserves the orientation of $Y$ and $b$ preserves the orientation of $X.$
\end{lem} 
\begin{proof}
For $\gamma_i\in C_{\rho}(\mathbb Z_p)C_{\tau}(\mathbb Z_p),$ let $\gamma_i=ab=a'b',$ and $a=ca'$ as in linear algebra lemma, and look at how the rational matrix $a$ acts on the line $L.$
Pick a non-zero vector $0\not= v\in L_{\mathbb Q}$ and note that $av=kv$ for some $k\in\mathbb Q.$
Also, recall that the group $\Gamma_Y$ consists of integer matrices preserving
the orientation of $Y$, so that $\Gamma_Yv=v.$  
Thus 
$$
\Gamma_Ya'v=\Gamma_Yac^{-1}v=kc^{-1}v.
$$

\begin{itemize}
\item
If $kc^{-1}\notin\mu$ then $\Gamma_Ya'\cap \mu U(p^n)=\emptyset$ for a sufficiently small,
$p$-adic neighborhood $\mu U(p^n)$ of the group $\mu$.
Thus, the product $\Gamma_Ya'\times b'\Gamma_X$ misses the $\mu$-invariant neighborhood of the identity $\mu U(p^n)\times C_{\tau}(\mathbb Z_p)$ in the Cartesian product $C_{\rho}(\mathbb Z_p)\times C_{\tau}(\mathbb Z_p)$ and consequently its image $\Gamma_Ya'b'\Gamma_X=\Gamma_Y\gamma_i\Gamma_X$ misses the 
neighborhood of the identity $U(p^n)C_{\tau}(\mathbb Z_p)$ in the product $C_{\rho}(\mathbb Z_p)C_{\tau}(\mathbb Z_p).$
This means $\Gamma(p^n)\cap\Gamma_Y\gamma_i\Gamma_X=\emptyset$ for sufficiently large $n$. For such $n$, the flats $\Gamma_Y\gamma_iX$ do not appear in $\Gamma(p^n)X$. 
\item
On the other hand, if $kc^{-1}=\nu$ for some $m$-th root of unity $\nu\in\mathbb Q_p$ then 
$$
{a\over k}={ca'\over c\nu}={a'\over \nu}
$$
is a rational matrix with determinant $\det(a'/\nu)=1$. In other words, $a/k\in C_{\rho}(\mathbb Q)$. Similarly, $kb\in C_{\tau}(\mathbb Q)$. Moreover $a/k$ preserves the orientation of $Y$ because $(a/k)v=v$ and $kb$ preserves the orientation of $X$ since it lies in the centralizer of $\tau$.
\end{itemize}
In conclusion, since there are only finitely many $\gamma_i$, we can pick $n$ sufficiently large so that every flat in $\Gamma(p^n)X$ that intersects $Y$ is of the form described in the statement of the Lemma. This completes the proof of the Lemma and thus also the proof of the main Claim.
\end{proof}

\section*{Proof of Theorem \ref{maintheorem}}
The condition ($t$) 
saying that the only solutions to $x\tau=\tau x,x\rho=\rho x$ are scalar matrices is satisfied whenever
the collection of eigenspaces $(L_1,\dots,L_m)$ of $\tau$ and the line-hyperplane pair $(L,P)$ of $\rho$ 
are in general position. 

Let $X_1,\dots,X_N$ and $Y_1,\dots, Y_N$ be the subspaces obtained in Proposition \ref{rationalintersectionpattern}.
The eigenspaces of $X_i$ and line-hyperplane pairs of $Y_j$ are in general position, so we can apply
Theorem \ref{covers} and find $n_0$ such that for $n\geq n_0$ the quotients $\overline{X_i}=X_i/\Gamma_{X_i}(p^n)$
and $\overline{Y_i}=Y_i/\Gamma_{Y_i}(p^n)$ 
\begin{itemize}
\item
are embedded orientable manifolds, 
\item
$\overline{X_i}$ and $\overline{Y_i}$
intersect
\item
all the intersections of $\overline{X_i}$ and $\overline{Y_i}$
are transverse and have the same sign,
\item
$\overline{X_i}$ and $\overline{Y_j}$ do not intersect for $i>j.$ 
\end{itemize}
This means the intersection matrix is upper triangular with non-zero diagonal entries.
Consequently, it is a non-degenerate $N\times N$ matrix, which means  
the flats $\overline{X_i}$ span an $N$-dimensional subspace of $H_{m-1}(H/\Gamma(p^n);\mathbb Q).$
\section{Questions}
We end the paper by mentioning some questions related to the results of this paper.
\begin{itemize}
\item(Other symmetric spaces)
Maximal periodic flats can be found in any locally symmetric space. For which symmetric spaces do they give non-trivial rational homology classes?
An argument analogous to the one presented in this paper can be performed for Hilbert modular surfaces, and we expect that there is a common generalization to lattices of the form $\SL_m\mathcal O_K$, where $K$ is a totally real number field. On the other hand, some lattices in complex hyperbolic (hence real rank one) spaces have Property (T), and for these the maximal tori are closed geodesics that do not give rationally nontrivial cycles (since Property (T) implies $H_1$ is torsion.)   
\item(Uniform lattices in $\SL_3$)
There are uniform lattices in $\SL_3$ (associated to units of division algebras) for which the only properly immersed totally geodesic subspaces are flat tori. For these, there are no candidate complementary totally geodesic subspaces. Do the $2$-tori in these lattices give torsion homology classes?
\item(Rates of growth)
How fast does the subspace of (rational) homology generated by maximal flat tori grow in congruence covers? This question is especially interesting for Hilbert modular surfaces: in this case, the second homology grow linearly in the degree of the cover. Does the space of $2$-tori also grow linearly?
\item(Peripheral cycles)
Do the maximal tori give homology cycles that come from the end? If one can find complementary {\it compact} cycles that intersect these tori non-trivially in homology, then the answer is no. In Hilbert modular surfaces, one can intersect with complementary $2$-tori, but already for the lattice $\SL_3\mathbb Z$ this question is open. 
\end{itemize}

\bibliography{locsym}
\bibliographystyle{alpha}

\end{document}